\documentclass[10pt]{article}
 \makeatletter
\newfont{\footsc}{cmcsc10 at 8truept}
\newfont{\footbf}{cmbx10 at 8truept}
\newfont{\footrm}{cmr10 at 10truept}
\renewcommand{\ps@plain}{%
\renewcommand{\@oddfoot}{\footsc {\footbf }  \footrm\thepage}}

\makeatother  \leftmargin=25mm
\usepackage{color}
\usepackage[centertags]{amsmath}
\usepackage{tikz}
\usepackage{pgfplots}
\usepackage{amsfonts}
\usepackage{amsthm}
\usepackage{newlfont}
\usepackage{graphicx}
\usepackage{epstopdf}
\usepackage{subfigure}
\usepackage[utf8]{inputenc}
\usepackage{indentfirst}
\usepackage{latexsym}
\usepackage{float}
\usepackage{cite}
\usepackage{CJK}
\usepackage{array}
\usepackage{amsfonts}
\usepackage{mathrsfs}
\usepackage{amssymb}
\usepackage{mathtools}
\usepackage[ruled,linesnumbered]{algorithm2e}
\usepackage{float}     
\usepackage[T1]{fontenc}
\usepackage{authblk}

\usepackage[colorlinks,
linkcolor=black,
anchorcolor=black,
citecolor=black]{hyperref}

\newtheorem{thm}{Theorem}[section]

\newtheorem{lem}[thm]{Lemma}

\newtheorem{prob}{Problem}
\theoremstyle{definition}
\newtheorem{defn}[thm]{Definition}
\newtheorem{case}{Case}
\theoremstyle{remark}
\newtheorem{rem}[thm]{Remark}

\title{Classification and enumeration of lattice polygons in a disc\footnotemark[2]\footnotetext[2]{This work is supported by the National Natural Science Foundation of China (NSFC12226006, NSFC11921001) and the National Key
Research and Development Program of China (2018-YFA0704701).}}
\date{}
\author{Qiuyue Liu\textsuperscript{1}, Yuqin Zhang\textsuperscript{2} and Zhanyuan Cai\thanks{Corresponding author. E-mail address: 18976257573@163.com.}}
\begin{document}
	\maketitle
\begin{center}\vskip -0.8cm
		{\small Center for Applied Mathematics,
			Tianjin University, 300072, Tianjin, China
			\vskip 0.3cm
		}
	\end{center}
\begin{abstract}
In 1980, V. I. Arnold studied the classification problem for convex lattice polygons of given area. Since then, this problem and its analogues have been studied by many authors, including $\mathrm{B\acute{a}r\acute{a}ny}$, Lagarias, Pach, Santos, Ziegler and Zong. Recently, Zong proposed two computer programs to prove Hadwiger's covering conjecture and Borsuk's partition problem, respectively, based on enumeration of the convex lattice polytopes contained in certain balls. For this purpose, similar to $\mathrm{B\acute{a}r\acute{a}ny}$ and Pach's work on volume and Liu and Zong's work on cardinality, we obtain bounds on the number of non-equivalent convex lattice polygons in a given disc. Furthermore, we propose an algorithm to enumerate these convex lattice polygons.
\\[3mm]
\\[3mm]
\end{abstract}

\maketitle
\section{Introduction}
Convex lattice polytopes in $\mathbb{E}^{d}$ are convex hulls of finite subsets of the integer lattice $\mathbb{Z}^{d}$. Alternatively, they can be described as the intersections of convex bodies and $\mathbb{Z}^{d}$, which is the definition used in this paper for convenience. As usual, let $P$ denote a $d$-dimensional convex lattice polytope, let $|P|$ denote the cardinality of $P$, let $f_{0}(P)$ denote the number of vertices of $P$, and let $v(P)$ denote the volume of $P$. For further information on polytopes and lattice polytopes, we can refer to \cite{Gruber} and \cite{Ziegler}.

Consider two $d$-dimensional convex lattice polytopes, denoted by $P_{1}$ and $P_{2}$. If there exists a $\mathbb{Z}^{d}$-preserving unimodular transformation $\sigma$ such that $\sigma(P_{1})=P_{2}$, we say that $P_{1}$ and $P_{2}$ are \emph{unimodularly equivalent}, denoted by $P_{1}\thicksim P_{2}$. It is easy to see that, if $P_{1}\thicksim P_{2}$ and $P_{2}\thicksim P_{3}$, then we have $P_{1}\thicksim P_{3}$. In addition, if $P_{1}\thicksim P_{2}$, then we have
\begin{align*}
    |P_{1}|=|P_{2}|,\ f_{0}(P_{1})=f_{0}(P_{2}),\ v(P_{1})=v(P_{2}).
\end{align*}
When two vertices are the endpoints of the same side in a lattice polytope, they are said to be adjacent. Therefore, if $P_{1}\thicksim P_{2}$ and $\sigma$ is the unimodular transformation, then $\sigma$ preserves both adjacency and non-adjacency between vertices. Furthermore, if the unimodular transformation $\sigma$ maps $v^{(i)}_{1}$ to $v^{(i)}_{2}$, where $v^{i}_{1}$ and $v^{i}_{2}$ are vertices of $P_{1}$ and $P_{2}$ respectively, then the volume of the convex hull formed by $v^{(1)}_{i}$ and its adjacent vertices is the same as that of the convex hull formed by $v^{(2)}_{i}$ and its adjacent vertices. Moreover, when $P_{1}\thicksim P_{2}$, for any side $S$ of $P_1$, the number of lattice points on $S$ is the same as that on $\sigma(S)$.

Clearly, convex lattice polytopes can be divided into different classes based on the equivalence relation $\thicksim$. By triangulations, it is easy to show that
\begin{equation}
    d!\cdot v(P)\in \mathbb{Z}
\end{equation}
holds for any $d$-dimensional convex lattice polytope $P$. Let $v(d,m)$ denote the number of different classes of $d$-dimensional convex lattice polytopes $P$ with $v(P)=m/d!$, where both $d$ and $m$ are positive integers. Let $f(d,m)$ and $g(d,m)$ be functions of positive integers $d$ and $m$. In this paper, $f(d,m)\ll g(d,m)$ means that
\begin{equation}
    f(d,m)\leqslant c_{d}\cdot g(d,m)
\end{equation}
holds for all positive integers $m$ and a fixed positive integer $d$, where $c_{d}$ is a constant that depends only on $d$.

In 1980, Arnold\cite{arnol1980statistics} studied the values of $v(2,m)$ and proved that
\begin{align}
    m^{\frac{1}{3}}\ll \log v(2,m)\ll  m^{\frac{1}{3}}\log m,
\end{align}
holds for sufficiently large $m$.  In 1991, Lagarias and Ziegler \cite{Lagarias} proved that any lattice polytope with volume at most $V$ is unimodularly equivalent to a lattice polytope contained in a lattice cube of side length at most $n\cdot n!V$, implying that the number of equivalence classes of lattice polytopes of bounded volume is finite. In 1992, $\mathrm{B\acute{a}r\acute{a}ny}$ and Pach \cite{barany} improved Arnold’s upper bound to 
\begin{align}\label{improvedjie}
\log v(2,m)\ll m^{\frac{1}{3}}. 
\end{align}
Furthermore, $\mathrm{B\acute{a}r\acute{a}ny}$ and Vershik \cite{Vershik} obtained a general upper bound 
\begin{align}
    \log v(d,m)\ll m^{\frac{d-1}{d+1}},
\end{align}
for any fixed $d$ and sufficiently large $m$.

Let $\kappa(d,w)$ denote the number of different classes of $d$-dimensional convex lattice polytopes $P$ with $|P|=w$. In 2011, Liu and Zong \cite{liu2011} studied  the classification problem for convex lattice polytopes of given cardinality by proving
\begin{equation}
    \begin{gathered}
        w^{\frac{1}{3}}\ll \log \kappa(2,w)\ll  w^{\frac{1}{3}},\\
        \kappa(d,w)=\infty,\ if\ w\geqslant d+1\geqslant 4.
    \end{gathered}
\end{equation}
This paper also obtains some relevant results of Arnold’s problem for centrally symmetric lattice polygons.

In \cite{Aliev}, I. Aliev, J. A. De Loera and Q. Louveaux presented six contributions to the study of $k$-feasibility for the semigroup $Sg(A)=\{\mathbf{b}: A\mathbf{x}=\mathbf{b}, x\in\mathbb{Z}^{n}, \mathbf{x}\geqslant \mathbf{0}\}$, where $A$ is an integral $d\times n$ matrix, and the associated polyhedral geometry. In fact, the $k$-feasibility questions are special cases (for parametric polyhedra of the form $\{\mathbf{x}: A\mathbf{x} = \mathbf{b}, \mathbf{x}\geqslant \mathbf{0}\}$) of the problem of classifying polyhedra with $k$ lattice points.

In addition to the above results about the number of classes, several authors have focused on enumerating $d$-dimensional convex lattice polytopes under particular conditions. S. Rabinowitz \cite{Rabinowitz} enumerated all representative elements of convex lattice polygons with at most one interior lattice point. M. Blanco and F. Santos \cite{Santos1} proved the finiteness result in dimension three for polytopes of width larger than one and a fixed number of lattice points, and performed an explicit enumeration up to $11$ lattice points \cite{Santos, Blanco}. In 2021, B. Gabriele\cite{Gabriele} developed an algorithm to completely enumerate representative elements of $d$-dimensional convex lattice polytopes with a volume at most $m$. Additionally, J.A. De Loera discussed algorithms and software for enumerating all lattice points inside a rational convex polytope in \cite{De Loera}, while in \cite{De Loera1}, he provided a survey for the problem of counting the number of lattice points inside a polytope and its applications.

From the above studies, it can be seen that the classification and enumeration of convex lattice polytopes have been considered only under conditions about the structure of convex lattice polytopes themselves, such as interior, volume, cardinality, width, etc. However, to our knowledge, few studies have focused on the classification and enumeration of convex lattice polytopes that are restricted to a given region, namely, the convex lattice polytopes formed by lattice points in a given region. This is a crucial and valuable area of research for the following reasons.

In 1957, Hadwiger \cite{Hadwiger} proposed a conjecture that every $n$-dimensional convex body can be covered by $2^n$
translates of its interior. This conjecture, known as Hadwiger's conjecture, is still open for all $n\geqslant 3$. In 1933, Borsuk \cite{Borsuk} put forth a conjecture that every $n$-dimensional bounded set can be divided into $n+1$ subsets of smaller diameters, which is called Borsuk's conjecture. In 2010, Zong\cite{C. Zong2} introduced a four-step quantitative program to tackle Hadwiger's conjecture. In 2021, Zong\cite{C. Zong1} developed a computer proof program to deal with Borsuk’s conjecture, based on a novel reformulation. Both of Zong's programs rely on the enumeration of convex lattice polytopes $P$ contained in certain balls.

In addition, the complexity analysis of the algorithm for enumerating all representative elements of convex lattice polytopes within a given region can contribute to determining whether this enumeration problem is NP-hard or NP-complete. This could lead to the discovery of new lattice-based NP-hard or NP-complete problems, which would be of great significance in the field of lattice cryptography.

In this paper, we focus on convex lattice polytopes in the 2-dimensional case, namely, convex lattice polygons. Based on the above motivations, in section 2, we study the number of different classes of convex lattice polygons within a given closed disc of radius $R$ centered at the origin, denoted as $B_{R}$.
Specifically, let $\kappa(R)$ denote the number of different classes of convex lattice polygons in $B_{R}$, and we prove
\begin{align*}
    R^{2/3}\ll \log \kappa(R)\ll R^{2/3}.
\end{align*}
 This result is analogous to the results of $\mathrm{B\acute{a}r\acute{a}ny}$ and Pach \cite{barany}, as well as Liu and Zong \cite{liu2011}.
 
 In section $3$ we present an algorithm to enumerate representative elements of convex lattice polygons in $B_{R}$ and obtain the number of different classes when $R=2$, 3 and 4. Notably, this algorithm is effective for any other regions besides the disc in the 2-dimensional case. Moreover, this algorithm can contribute to the computer proof in \cite{C. Zong1}.
 
 In section 4 we analyze the experimental results obtained from the algorithm presented in section 3. Based on our observations, we propose two open problems aimed at enhancing our comprehension of the number of representative elements in $B_{R}$. 
 
In section 5 we analyze the complexity of the algorithms. This contributes to determining whether the problem of finding all non-equivalent convex lattice polygons in $B_{R}$ is NP-hard or NP-complete with respect to $R$, and whether the problem of finding all non-equivalent convex lattice polytopes in $\alpha B^{n}$ is is NP-hard or NP-complete with respect to $n$, where $B^{n}$ is the $n$-dimensional unit ball centered at the origin in $\mathbb{E}^{n}$.

\section{Classification of Convex Lattice Polygons in a disc}
Recall that $\kappa(R)$ denotes the number of all different classes of convex lattice polygons in $B_{R}$. The main task of this section is to prove the following theorem.
\begin{thm}\label{bound}
	When $R$ is sufficiently large, we have
	\begin{equation}
		c_1R^{2/3}\leqslant \log \kappa(R)\leqslant c_2R^{2/3},
	\end{equation}
	for some suitable $c_1,\ c_2>0$. Namely, we have
	\begin{equation}
		R^{2/3}\ll \log \kappa(R)\ll R^{2/3}.
	\end{equation}
\end{thm}

Before proving Theorem \ref{bound}, we first, as done in \cite{liu2011}, construct convex lattice polygons $M_\tau$ and $Q_i$, $i=1,\cdots,2^{\left|V_\tau\right|-1}$. The idea of this construction was originally derived from Arnold \cite{arnol1980statistics}.

Let $\tau$ be a large number, let $V_{\tau}$ denote the set of all primitive integer vectors in the semicircle $\{(x,y):\ x^2+y^2\leqslant \tau^2,\ x>0\}$. Let $M_\tau$ be the convex lattice polygon whose oriented sides are all integral vectors in $V_\tau$ and $-\sum_{v\in V_\tau}v$, as shown in Figure \ref{vtauandmtau}.

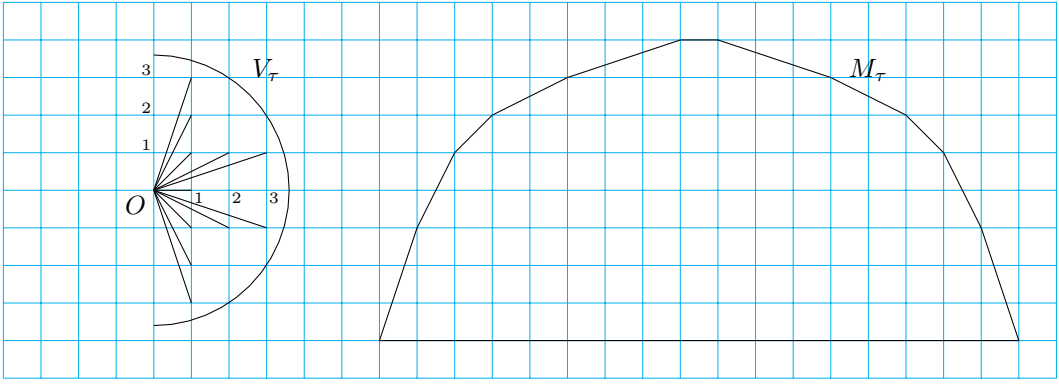
\begin{figure}
\centering
\begin{tikzpicture}
	\draw[step=.5cm,cyan,very thin] (-7,-2.5) grid (7,2.5);
	\draw [domain=-90:90] plot ({-5+1.8*cos(\x)}, {1.8*sin(\x)});
	\path (-5.24,-0.2) node (o) {$O$};
	\path[font=\fontsize{6}{6}\selectfont] (-4.4,-0.1) node (a) {1};
	\path[font=\fontsize{6}{6}\selectfont] (-3.9,-0.1) node (b) {2};
	\path[font=\fontsize{6}{6}\selectfont] (-3.4,-0.1) node (c) {3};
	
	\path[font=\fontsize{6}{6}\selectfont] (-5.1,0.6) node (a') {1};
	\path[font=\fontsize{6}{6}\selectfont] (-5.1,1.1) node (b') {2};
	\path[font=\fontsize{6}{6}\selectfont] (-5.1,1.6) node (c') {3};
	
	\draw[shift={(-1,0)}] (-4,0) -- (-3.5,0);
	\draw[shift={(-1,0)}] (-4,0) -- (-3.5,0.5);
	\draw[shift={(-1,0)}] (-4,0) -- (-3.5,1);
	\draw[shift={(-1,0)}] (-4,0) -- (-3.5,1.5);
	\draw[shift={(-1,0)}] (-4,0) -- (-3,0.5);
	\draw[shift={(-1,0)}] (-4,0) -- (-2.5,0.5);
	
	\draw[shift={(-1,0)}] (-4,0) -- (-3.5,-0.5);
	\draw[shift={(-1,0)}] (-4,0) -- (-3.5,-1);
	\draw[shift={(-1,0)}] (-4,0) -- (-3.5,-1.5);
	\draw[shift={(-1,0)}] (-4,0) -- (-3,-0.5);
	\draw[shift={(-1,0)}] (-4,0) -- (-2.5,-0.5);
	
	\path[font=\fontsize{10}{10}\selectfont] (-3.5,1.6) node (vtau) {$V_{\tau}$};
	
	\draw[shift={(6,2)}] (-4,0) -- (-3.5,0);
	\draw[shift={(3,0.5)}] (-4,0) -- (-3.5,0.5);
	\draw[shift={(2.5,-0.5)}] (-4,0) -- (-3.5,1);
	\draw[shift={(2,-2)}] (-4,0) -- (-3.5,1.5);
	\draw[shift={(3.5,1)}] (-4,0) -- (-3,0.5);
	\draw[shift={(4.5,1.5)}] (-4,0) -- (-2.5,0.5);
	
	\draw[shift={(9,1)}] (-4,0) -- (-3.5,-0.5);
	\draw[shift={(9.5,0.5)}] (-4,0) -- (-3.5,-1);
	\draw[shift={(10,-0.5)}] (-4,0) -- (-3.5,-1.5);
	\draw[shift={(8,1.5)}] (-4,0) -- (-3,-0.5);
	\draw[shift={(6.5,2)}] (-4,0) -- (-2.5,-0.5);
	\draw (-2,-2) -- (6.5,-2);
	
	\path[font=\fontsize{10}{10}\selectfont] (4.5,1.6) node (mtau) {$M_{\tau}$};
\end{tikzpicture}
\caption{Examples of $V_\tau$ and $M_\tau$}
\label{vtauandmtau}
\end{figure}



It is well known in number theory that
\begin{equation}\label{vtaushu}
	\left|V_\tau\right|=\frac{3}{\pi}\tau^2+O(\tau\log\tau).
\end{equation}
 We list some properties of the convex lattice polygon $M_\tau$ as follows:
\begin{enumerate}
	\item It has $\left|V_\tau\right|+1$ vertices.
	\item\label{property2} Apart from the diameter, each side of $M_\tau$ contains no other lattice points except the vertices.
    \item\label{property3} Let $r$ be the largest radius of semicircles contained in $M_{\tau}$ and let $r'$ be the smallest radius of semicircles containing $M_{\tau}$. By ($\ref{vtaushu}$) it can be deduced that
    \begin{align*}
        \tau^{3}\ll r\leqslant r'\ll \tau^{3}.
    \end{align*}
\end{enumerate}

Now we consider the convex lattice polygon $2M_\tau$. Let $S_1$, $S_2$, $\cdots$, $S_{\left|V_\tau\right|}$ be the short sides of $2M_\tau$ in the clockwise order. By property \ref{property2}, for each $i$, $S_i$ contains three lattice points, which are denoted by $\mathbf{q}_i^0$, $\mathbf{q}_i^1$ and $\mathbf{q}_i^2$ respectively. Then we have  $\mathbf{q}_1^0=(0,0)$, $\mathbf{q}_i^0=\mathbf{q}_{i-1}^2$ and $\mathbf{q}_{\left|V_\tau\right|}^2=2\sum_{\mathbf{v}\in V_\tau}\mathbf{v}$. For each $\mathbf{u}=(i_{1},i_{2},...,i_{\left|V_{\tau}\right|-1}) \in \{1,2\}^{\left|V_{\tau}\right|-1}$, we get a convex lattice polygon
\begin{equation*}
	Q_\mathbf{u}=\mathrm{conv}\{\mathbf{q}^{0}_{1},\mathbf{q}^{i_{1}}_{1},\mathbf{q}^{i_{2}}_{2},\cdots,\mathbf{q}^{i_{|V_{\tau}|-1}}_{|V_{\tau}|-1},\mathbf{q}^{2}_{\left|V_{\tau}\right|}\}.
\end{equation*}
Obviously when $\mathbf{u,v}\in \{1,2\}^{|V_{\tau}|-1}$ and $\mathbf{u}\neq \mathbf{v}$, $Q_\mathbf{u}$ and $Q_\mathbf{v}$ are different. Let $\mathcal{Q}$ be the set of these polygons $Q_\mathbf{u}$, $\mathbf{u}\in \{1,2\}^{|V_{\tau}|-1}$. It is obvious that $\left|\mathcal{Q}\right|=2^{\left|V_\tau\right|-1}$, and we have the following theroem.

\begin{thm}\label{noequi}
    For any $Q_\mathbf{u}\in \mathcal{Q}$, there is at most another $Q_\mathbf{v}\in \mathcal{Q}$ unimodularly equivalent to $Q_\mathbf{u}$.
\end{thm}
\begin{proof}
	According to the construction of $2M_{\tau}$, we have
 \begin{align}\label{2mtaubaohanyu}
   2M_\tau\subseteq \{(x,y):x\geqslant0,y\geqslant0\}.
 \end{align}
 Given any $Q_\mathbf{u}\in \mathcal{Q}$, suppose there exists a convex lattice polygon $Q_\mathbf{v}\in\mathcal{Q}$ unimodularly equivalent to $Q_\mathbf{u}$, that is, there exists a unimodular transformation $\sigma$ such that $\sigma(Q_\mathbf{u})=Q_\mathbf{v}$. Now consider $Q_\mathbf{u}$ as a matrix, whose columns are the coordinates of the lattice points contained in polygon $Q_\mathbf{u}$, and do the same for $Q_\mathbf{v}$. Then the transformation $\sigma$ can be written as follows:
	\begin{equation*}
		\sigma(Q_\mathbf{u})\coloneqq AQ_\mathbf{u}+b=Q_\mathbf{v},
	\end{equation*}
	where $A$ is a unimodular matrix and $b$ is an integer vector.
	
	Let $q_\mathbf{u}^{(1)}=(0,0)$, $q_\mathbf{u}^{(2)}$, $\cdots$, $q_\mathbf{u}^{(\left|V_\tau\right|+1)}$ be the vertices of $Q_\mathbf{u}$ in the clockwise order, and we denote the vertices of $Q_\mathbf{v}$ in the same way. By the definition of $Q_\mathbf{u}$ and $Q_\mathbf{v}$, $q_\mathbf{u}^{(\left|V_\tau\right|+1)}=q_\mathbf{v}^{(\left|V_\tau\right|+1)}=(n,0)$. It is easy to see that, for both $Q_\mathbf{u}$ and $Q_\mathbf{v}$, any short side contains strictly fewer lattice points than the diameter. So $\sigma$ maps the diameter of $Q_\mathbf{u}$ to the diameter of $Q_\mathbf{v}$. Therefore, the discussion can be divided into the following two cases.
	
	\begin{case}\label{case1}
		$\sigma(q_\mathbf{u}^{(1)})=q_\mathbf{v}^{(1)}$, $\sigma(q_\mathbf{u}^{(\left|V_\tau\right|+1)})=q_\mathbf{v}^{(\left|V_\tau\right|+1)}$. By the fact that unimodular transformation preserves the adjacency of the vertices of convex lattice polygons, we have $\sigma(q_\mathbf{u}^{(i)})=q_\mathbf{v}^{(i)}$ for all $i\in\{1,\cdots,\left|V_\tau\right|+1\}$. It is easy to see that $b=(0,0)$ and $A$ can be written in the following form:
		\begin{equation*}
			A=
			\begin{bmatrix}
				-1 & a\\
				0 & 1
			\end{bmatrix}
			\ or\ 
			\begin{bmatrix}
				-1 & a\\
				0 & -1
			\end{bmatrix},
		\end{equation*}
		where $a$ is an integer. By (\ref{2mtaubaohanyu}), we know that $Q_\mathbf{u},Q_\mathbf{v}\subseteq \{(x,y):x\geqslant0,y\geqslant0\}$, then $A$ can only be 
  \begin{align*}
      A=
			\begin{bmatrix}
				-1 & a\\
				0 & 1
			\end{bmatrix}.
  \end{align*}
  
  If there exists an $i\in\{2,\cdots,\left|V_\tau\right|\}$ such that $q_\mathbf{u}^{(i)}=q_\mathbf{v}^{(i)}$, then $A=I_2$, where $I_2$ is the identity matrix. Otherwise $q_\mathbf{u}^{(2)}\neq q_\mathbf{v}^{(2)}$ (in fact now $q_\mathbf{u}^{(i)}\neq q_\mathbf{v}^{(i)}$ for all $i\in\{2,\cdots,\left|V_\tau\right|\}$), which means that $q_\mathbf{u}^{(2)}$ and $q_\mathbf{v}^{(2)}$ have different vertical coordinates (when $\tau\geqslant \sqrt{5}$). This is because $q_\mathbf{u}^{(2)}$ and $q_\mathbf{v}^{(2)}$ can only be lattice point $\mathbf{q}_{1}^{1}$ or $\mathbf{q}_{1}^{2}$ on the short side $S_{1}$ of $2M_\tau$. But for any integer point $z\in\mathbb{Z}^2$, $Az$ and $z$ have the same vertical coordinates, as do $q_\mathbf{u}^{(2)}$ and $\sigma(q_\mathbf{u}^{(2)})= q_\mathbf{v}^{(2)}$, which leads to a contradiction. So in this case, we have $A=I_2$ and $Q_\mathbf{v}=I_2Q_\mathbf{u}=Q_\mathbf{u}$.
	\end{case}
	
	\begin{case}\label{case2}
		$\sigma(q_\mathbf{u}^{(1)})=q_\mathbf{v}^{(\left|V_\tau\right|+1)}$, $\sigma(q_\mathbf{u}^{(\left|V_\tau\right|+1)})=q_\mathbf{v}^{(1)}$. It is easy to see that
		\begin{equation*}
			A=
			\begin{bmatrix}
				-1 & a\\
				0 & 1
			\end{bmatrix}
			\ or\ 
			\begin{bmatrix}
				-1 & a\\
				0 & -1
			\end{bmatrix},\ 
			b=
			\begin{bmatrix}
				n\\0
			\end{bmatrix},
		\end{equation*}
		where $a$ is an integer. Similar to case \ref{case1}, $A$ can only be 
  \begin{align*}
      A=
			\begin{bmatrix}
				-1 & a\\
				0 & 1
			\end{bmatrix}.
  \end{align*}
  
		Let
		\begin{equation}\label{qupie}
			Q_\mathbf{u}'\coloneqq
			\begin{bmatrix}
				-1 & 0\\
				0 & 1
			\end{bmatrix}
			Q_\mathbf{u}+
			\begin{bmatrix}
				n\\0
			\end{bmatrix}
                \coloneqq \sigma'(Q_\mathbf{u}),
		\end{equation}
		we have
		\begin{equation}\label{lingygeguanxi}
			Q_\mathbf{v}=
			\begin{bmatrix}
				-1 & a\\
				0 & 1
			\end{bmatrix}
			Q_\mathbf{u}+
			\begin{bmatrix}
				n\\0
			\end{bmatrix}
			=
			\begin{bmatrix}
				1 & a\\
				0 & 1
			\end{bmatrix}
			Q_\mathbf{u}'\coloneqq A'Q_\mathbf{u}'.
		\end{equation}

        Obviously, $\sigma'$ in (\ref{qupie}) flips $Q_\mathbf{u}$ about Y-axis and translates it to the right by $n$ units. Let $q_\mathbf{u}^{\prime(1)}=\sigma'(q_\mathbf{u}^{(\left|V_\tau\right|+1)})=(0,0)$ and $q_\mathbf{u}^{\prime(i)}$ be the $i$th vertex of $Q'_\mathbf{u}$ in the clockwise order, then
            \begin{gather}
                q_\mathbf{u}^{\prime(i)}=\sigma'(q_\mathbf{u}^{(\left|V_\tau\right|+2-i)}),\quad i\in\{1,\cdots,\left|V_\tau\right|+1\}.
            \end{gather}
        Particularly $q_\mathbf{u}^{\prime(\left|V_\tau\right|+1)}=\sigma'(q_\mathbf{u}^{(1)})=(n,0)$.
        
        According to (\ref{lingygeguanxi}), we know that $A^{'}$  maps vertex $(0,0)$ of $Q_{\mathbf{u}}^{'}$ to vertex $(0,0)$ of $Q_{\mathbf{v}}$. Further, $A^{'}$  maps vertex $(n,0)$ to vertex $(n,0)$. So similar to  case \ref{case1}, we have $A'=I_2$ and $\tau\geqslant\sqrt{5}$. So in this case, we have $a=0$,
		\begin{equation*}
			A=
			\begin{bmatrix}
				-1 & 0\\
				0 & 1
			\end{bmatrix}\ 
			and\ Q_\mathbf{v}=
			\begin{bmatrix}
				-1 & 0\\
				0 & 1
			\end{bmatrix}
			Q_\mathbf{u}+
			\begin{bmatrix}
				n\\0
			\end{bmatrix}.
		\end{equation*}
	\end{case}
	The theorem has been proved.
\end{proof}

To further simplify the proof of Theorem \ref{bound}, we state the following lemma.

\begin{lem}\label{dandR}
	Given a sufficiently large $R$ and the corresponding closed disc $B_{R}$, let $\tau\geqslant\sqrt{2}$, $d$ be the length of the diameter of $2M_{\tau}$. If 
	\begin{equation}
		d\leqslant \sqrt{2}R,
	\end{equation}
	then $2M_\tau$ can be covered by $B_R$ after translating $d/2$ units to the left.
\end{lem}

\begin{proof}
Given any $\tau\geqslant \sqrt{2}$, a convex lattice polygon $2M_{\tau}$ can be constructed. Let $d$ be the length of the diameter of $2M_\tau$, and $G_{\tau}$ denote the set of all primitive integer vectors in the domain $\{(x,y):x^2+y^2\leqslant\tau^2,x>0,y>0\}$. The vectors in $G_{\tau}$ are symmetric with respect to the line $y=x$. Let $2G_\tau\coloneqq \{2\mathbf{v}:\mathbf{v}\in G_\tau\}$, then we have
 \begin{align}\label{he}
     \sum_{\mathbf{v}\in 2G_{\tau}}=(\ell_{\tau},\ell_{\tau}),
 \end{align}
 where $\ell_{\tau}$ is a positive integer.

 The distance from the highest point $B$ of $2M_{\tau}$ to the X-axis can be calculated as the sum of the vertical coordinates of the vectors in $2G_{\tau}$, i.e. $\ell_{\tau}$. By (\ref{he}), we have $\ell_{\tau}<d/2$. Therefore, $2M_\tau$ can be covered by a lattice rectangle $T_\tau$ of length $d$ and width $d/2$ (it is a positive integer), of which the long side coincides with the diameter of the polygon $2M_\tau$, as shown on the left side of Figure \ref{XandY}.

 Supposing that the midpoint of $2M_{\tau}$'s diameter is $C$, we translate point $C$ to the origin by moving it to the left by $d/2$ units. This translation is also applied to the rectangle $T_{\tau}$ and convex lattice polygon $2M_{\tau}$, as shown on the right side of Figure \ref{XandY}. After the translation, the distance between the vertex $A$ (see the right side of Figure \ref{XandY}) of the rectangle and the origin is $\frac{\sqrt{2}}{2}d$. To make rectangle $T_{\tau}$ covered by $B_{R}$, $d$ only needs to satisfy
  \begin{align*}
      \frac{\sqrt{2}}{2}d\leqslant R,
  \end{align*}
  namely
  \begin{align}
      d\leqslant \sqrt{2}R.
  \end{align}
Thus, when $d\leqslant \sqrt{2}R$, $2M_{\tau}$ can be covered by $B_{R}$ after translating it to the left by $d/2$ units.
\end{proof}

\begin{figure}
\centering
    \begin{tikzpicture}
	\draw[step=.2cm,cyan,very thin] (-7,-2.4) grid (8.4,2.4);
	
	\draw[shift={(-2.4,0.4)}][shift={(2.2,1.2)},scale=0.8] (-4,0) -- (-3.5,0);
	\draw[shift={(-2.4,0.4)}][shift={(-0.2,0)},scale=0.8] (-4,0) -- (-3.5,0.5);
	\draw[shift={(-2.4,0.4)}][shift={(-0.6,-0.8)},scale=0.8] (-4,0) -- (-3.5,1);
	\draw[shift={(-2.4,0.4)}][shift={(-1,-2)},scale=0.8] (-4,0) -- (-3.5,1.5);
	\draw[shift={(-2.4,0.4)}][shift={(0.2,0.4)},scale=0.8] (-4,0) -- (-3,0.5);
	\draw[shift={(-2.4,0.4)}][shift={(1,0.8)},scale=0.8] (-4,0) -- (-2.5,0.5);
	
	\draw[shift={(-2.4,0.4)}][shift={(4.6,0.4)},scale=0.8] (-4,0) -- (-3.5,-0.5);
	\draw[shift={(-2.4,0.4)}][shift={(5,0)},scale=0.8] (-4,0) -- (-3.5,-1);
	\draw[shift={(-2.4,0.4)}][shift={(5.4,-0.8)},scale=0.8] (-4,0) -- (-3.5,-1.5);
	\draw[shift={(-2.4,0.4)}][shift={(3.8,0.8)},scale=0.8] (-4,0) -- (-3,-0.5);
	\draw[shift={(-2.4,0.4)}][shift={(2.6,1.2)},scale=0.8] (-4,0) -- (-2.5,-0.5);
	\draw (-6.6,-1.6) -- (0.2,-1.6);
	
	\draw[->] (-6.6,-1.6) -- (-6.6,2.1);
	\draw[->] (-6.6,-1.6) -- (0.5,-1.6);
	\path[font=\fontsize{6}{6}\selectfont] (-3.2,1.4) node (B) {$B$};
	\draw[fill=black] (-3.2,1.6) circle (1pt);
	\path[font=\fontsize{6}{6}\selectfont] (-3.2,-1.8) node (C) {$C$};
	\path[font=\fontsize{6}{6}\selectfont] (-6.7,-1.8) node (O) {$O$};
	\draw[fill=black] (-3.2,-1.6) circle (1pt);
	\path[font=\fontsize{8}{8}\selectfont] (-1.8,1.45) node (mmtau) {$2M_{\tau}$};
	\draw (-6.6,-1.6) rectangle (0.2,1.8);
	\path[font=\fontsize{8}{8}\selectfont] (-1.6,2.05) node (ttau) {$T_{\tau}$};
	
	\draw[shift={(5.4,0.4)}][shift={(2.2,1.2)},scale=0.8] (-4,0) -- (-3.5,0);
	\draw[shift={(5.4,0.4)}][shift={(-0.2,0)},scale=0.8] (-4,0) -- (-3.5,0.5);
	\draw[shift={(5.4,0.4)}][shift={(-0.6,-0.8)},scale=0.8] (-4,0) -- (-3.5,1);
	\draw[shift={(5.4,0.4)}][shift={(-1,-2)},scale=0.8] (-4,0) -- (-3.5,1.5);
	\draw[shift={(5.4,0.4)}][shift={(0.2,0.4)},scale=0.8] (-4,0) -- (-3,0.5);
	\draw[shift={(5.4,0.4)}][shift={(1,0.8)},scale=0.8] (-4,0) -- (-2.5,0.5);
	
	\draw[shift={(5.4,0.4)}][shift={(4.6,0.4)},scale=0.8] (-4,0) -- (-3.5,-0.5);
	\draw[shift={(5.4,0.4)}][shift={(5,0)},scale=0.8] (-4,0) -- (-3.5,-1);
	\draw[shift={(5.4,0.4)}][shift={(5.4,-0.8)},scale=0.8] (-4,0) -- (-3.5,-1.5);
	\draw[shift={(5.4,0.4)}][shift={(3.8,0.8)},scale=0.8] (-4,0) -- (-3,-0.5);
	\draw[shift={(5.4,0.4)}][shift={(2.6,1.2)},scale=0.8] (-4,0) -- (-2.5,-0.5);
	\draw (1.2,-1.6) -- (8,-1.6);
	
	\draw[->] (4.6,-1.6) -- (4.6,2.1);
	\draw[->] (1.2,-1.6) -- (8.3,-1.6);
	\path[font=\fontsize{6}{6}\selectfont] (1,1.9) node (A) {$A$};
	\draw[fill=black] (1.2,1.8) circle (1pt);
	\path[font=\fontsize{6}{6}\selectfont] (4.8,1.4) node (B') {$B$};
	\draw[fill=black] (4.6,1.6) circle (1pt);
	\path[font=\fontsize{6}{6}\selectfont] (4.6,-1.8) node (C') {$O(C)$};
	\draw[fill=black] (4.6,-1.6) circle (1pt);
	\path[font=\fontsize{8}{8}\selectfont] (6,1.45) node (mmtau') {$2M_{\tau}$};
	\draw (1.2,-1.6) rectangle (8,1.8);
	\path[font=\fontsize{8}{8}\selectfont] (6.4,2.05) node (ttau') {$T_{\tau}$};
\end{tikzpicture}
    \caption{Left: $2M_{\tau}$ and $T_{\tau}$ before translating; Right: translated $2M_{\tau}$ and $T_{\tau}$}
    \label{XandY}
\end{figure}
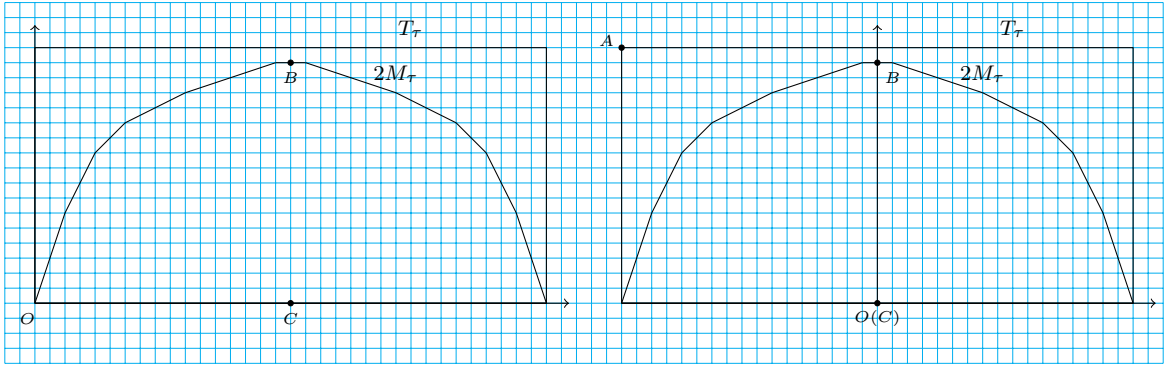

Now we are ready to prove Theorem \ref{bound}.

\begin{proof}[Proof of Theorem \ref{bound}]
Let $R$ be a sufficiently large number, let $B_{R}$ be a close disc as defined above. Firstly, we prove the upper bound. By the result of $\mathrm{B\acute{a}r\acute{a}ny}$ and Pach \cite{barany}, and the obvious fact that the area of each convex lattice polygon contained in $B_{R}$ does not exceed $\pi R^2$, we have
	\begin{align}\label{shangjie}
        \log\kappa(R)&\leqslant\log\sum_{i=1}^{\lfloor2\pi R^2\rfloor}v(2,i)\leqslant c_{2}R^{2/3},
	\end{align}
	for some  constant $c_{2}$.
	
	Next, we prove the lower bound. Let $d$ be the length of the diameter of $2M_\tau$, for any $\tau\geqslant\sqrt{5}$, it is easy to see that
	\begin{equation*}
		d\leqslant 2\left|V_{\tau}\right|\cdot\tau.
	\end{equation*}
	By (\ref{vtaushu}), for any constant $c>1$,
        \begin{equation*}\label{eqdandtau}
		d\leqslant c\frac{6}{\pi}\tau^{3}
	\end{equation*}
        will hold when $\tau$ is sufficiently large. So there exists a sufficiently large $\tau'$ such that when $\tau>\tau'$, we have
         \begin{align}
             d\leqslant \sqrt{2}\frac{6}{\pi}\tau^{3}.
         \end{align}

 Given a sufficiently large $R$, which satisfies
	\begin{equation}\label{tauxing}
		\tau^*\coloneqq \left(\frac{\pi}{6}R\right)^{1/3}>\tau',
	\end{equation}
	then we have
	\begin{equation}
		d\leqslant \sqrt{2}\frac{6}{\pi}\tau^{*3}=\sqrt{2}R.
	\end{equation}
    It is deduced by lemma \ref{dandR} that $2M_\tau^*$ can be covered by $B_R$ after an integral translation. Denote the translated polygon as $2M_{\tau^*}^{'}$, and we have $2M_{\tau^*}^{'}\subseteq B_R$.
	
	Now, we construct the set of convex lattice polygons in $2M_{\tau^*}^{'}$, denoted by $\mathcal{Q}$, following the procedure described above. It is obvious that each polygon in $\mathcal{Q}$ is contained in $B_R$. So by Theorem \ref{noequi}, we have
	\begin{equation}\label{kappaandvtau}
		\kappa(R)\geqslant \frac{\left|\mathcal{Q}\right|}{2}=2^{\left|V_{\tau^*}\right|-2}.
	\end{equation}
	By (\ref{vtaushu}), (\ref{tauxing}) and (\ref{kappaandvtau}), we get
	\begin{equation}
		\log\kappa(R)\geqslant \left|V_{\tau^*}\right|-2 \geqslant \frac{3}{\pi}\tau^{*2}=\frac{3}{\pi}(\frac{\pi}{6}R)^{2/3}=\frac{1}{2}(\frac{6}{\pi})^{1/3}R^{2/3},
	\end{equation}
	for any sufficiently large $R$.
\end{proof}

\begin{rem}\label{rem1}
If $R$ is sufficiently large, the constant $c_{2}$ given by (\ref{shangjie}) is not too large, that is, $\log\kappa(R)<11(2\pi R)^{2/3}$.
\end{rem}

By the proof of Theorem \ref{bound} and Remark \ref{rem1}, we can get
\begin{equation}
	\frac{1}{2}(\frac{6}{\pi})^{1/3}\leqslant\liminf_RR^{-2/3}\log\kappa(R)\leqslant\limsup_RR^{-2/3}\log\kappa(R)\leqslant 11(2\pi)^{2/3},
\end{equation}
which naturally leads to the following problem.

\begin{prob}
	Decide whether
	\begin{equation}\label{limr2/3}
		\lim_{R\rightarrow\infty}R^{-2/3}\log\kappa(R)
	\end{equation}
	exists or not. Determine the limit of (\ref{limr2/3}) if it exists.
\end{prob}

Furthermore, based on the previous construction of $\mathcal{Q}$ in $2M_{\tau^*}^{'}$, for the maximum number of vertices of convex lattice polygons in $B_{R}$, denoted by $M$, we have
\begin{align*}
    M\geqslant 2|V_{\tau^{*}}|.
\end{align*}
Namely, by (\ref{vtaushu}) and (\ref{tauxing}), we can get
\begin{align}
     M\geqslant\frac{6}{\pi}(\tau^{*})^{2}=\frac{6}{\pi}(\frac{\pi}{6}R)^{2/3}=(\frac{6}{\pi})^{1/3}R^{2/3}.
\end{align}
By the result in \cite{arnol1980statistics}, we know that
\begin{align}\label{dingdianshang}
    M\leqslant 16(2\pi R^2)^{1/3}=16(2\pi)^{1/3}R^{2/3}.
\end{align}

\section{Enumeration of Convex Lattice Polygons in a disc}
In this section, we introduce an algorithm to find  all representative elements of convex lattice polygons in a given disc, and then calculate the number of different classes. Here we classify the convex lattice polygons according to the cardinality.  Our idea is analogous to that of G. Brown and A. M. Kasprzyk \cite{boxshave}. We improve and extend the algorithm in \cite{boxshave}, and apply it to solve our problem.

\begin{defn}\cite{boxshave}
	Let $P$ be a convex lattice polygon, and let $v\in \mathrm{vert}(P)$ be a vertex of $P$. if the polygon $P_v\coloneqq \mathrm{conv}((P\cap \mathbb{Z}^2)\backslash \left\{v\right\})$ satisfies $\mathrm{dim} P_v=\mathrm{dim}P$, then $P_v$ is said to have been obtained from $P$ by \emph{shaving}. Given a polygon $Q$, we say that $Q$ can be obtained from $P$ via \emph{successive shaving} if there exists a sequence of shavings $Q=P^{(0)}\subset \cdots \subset P^{(n)}=P$, where $P^{(i-1)}=P^{(i)}_{v_i}$ for some $v_i\in \mathrm{vert}(P^{(i)}), 0<i\leqslant n$.
\end{defn}

The following lemma guarantees that all convex lattice polygons in a disc can be obtained by successive shaving vertices. The lemma is similar to the result of \cite{boxshave}. For the sake of completeness, we reproduce the proof here.
\begin{lem}\label{allpolyincir}
	Let $P$ be a convex lattice polygon, let $Q$ be a convex lattice polygon contained in $P$. Then $Q$ can be obtained from $P$ via successive shaving.
\end{lem}

\begin{proof}
	If $P=Q$, then we are done. If $Q\subsetneq P$, there must be a vertex $v\in \mathrm{vert}(P)$ such that $v\not\in Q$. By shaving $v$ from $P$ we can get a polygon $P_v$ satisfying $Q\subset P_v$. If $Q=P_v$, we are done. Otherwise, if $P_v\neq Q$, we can keep shaving vertices of $P$ and let $Q$ be contained in the new polygons.
	
According to the above definition we know
	\begin{equation*}
		\left|P^{(i-1)}\cap \mathbb{Z}^2\right|=\left|P^{(i)}\cap \mathbb{Z}^2\right|-1,\quad 0<i\leqslant n.
	\end{equation*}
	Because the polygons obtained by this shaving process always contain $Q$, the numbers of lattice points of those polygons have a lower bound, namely $\left|Q\cap \mathbb{Z}^2\right|$. So the process will terminate after finitely many steps of shaving, and the polygon obtained from the last step will be $Q$.
\end{proof}

Now we denote the largest convex lattice polygon within $B_R$ by $N_R$, namely,
\begin{equation}
	N_R\coloneqq \mathrm{conv}(B_R\cap \mathbb{Z}^2).
\end{equation}
Let $\mathcal{P}^k$ ($3\leqslant k\leqslant |N_{R}|$) denote the family of representative elements of convex lattice polygons with $k$ lattice points in disc $B_{R}$.  The following lemma will show that it is not necessary to find all convex lattice polygons in a disc for our purpose.
\begin{lem}\label{noneqisenough}
	Suppose that $\mathcal{P}^k\coloneqq\left\{P_1, \cdots  P_{n_k}\right\}$ is not empty for some $k$ $(4\leqslant k\leqslant |N_{R}|)$. Then for any convex lattice polygon $Q$ with $k-1$ lattice points which is contained in $B_R$, there exists $P_i\in \mathcal{P}^k$ and a vertex $v\in \mathrm{vert}(P_i)$, such that $Q$ is unimodularly equivalent to $\mathrm{conv}((P_i\cap\mathbb{Z}^{2})\backslash\{v\})$.
\end{lem}

\begin{proof}
	Let $\mathcal{P}^k\neq\emptyset$ and $Q$ be a convex lattice polygon with $k-1$ lattice points which is contained in $B_R$. Then there exists a sequence of shavings $Q=P^{(0)}\subset \cdots \subset P^{(n)}=N_R$, where $P^{(i-1)}=P^{(i)}_{v_i}$ for some $v_i\in \mathrm{vert}(P^{(i)}), 0<i\leqslant n$.
	
	In the above sequence of shavings, it is obvious that $P^{(1)}$ contain $k$ lattice points and $P^{(1)}\subset B_R$. So by the definition of $\mathcal{P}^k$, there exists $P_i\in\mathcal{P}^k$ such that $P_i$ is unimodularly equivalent to $P^{(1)}$. If $v\in \mathrm{vert}(P_i)$ is the image of $v_1\in P^{(1)}$ under the unimodular transformation between $P_i$ and $P^{(1)}$, then $Q$ is unimodularly equivalent to $\mathrm{conv}((P_i\cap \mathbb{Z}^{2})\backslash\{v\})$, which is what we need.
\end{proof}

Lemma \ref{allpolyincir} and \ref{noneqisenough} present the main idea of our algorithm. Starting with $n=\left|N_R\right|$, each vertex for every convex lattice polygon in $\mathcal{P}^n$ is shaved in turn to obtain the set of convex lattice polygons with $n-1$ lattice points. Then $\mathcal{P}^{n-1}$ can be obtained from this resulting set by removing equivalent convex lattice polygons. Finally, $\mathcal{P}^{3}\cup\cdots \cup \mathcal{P}^{|N_{R}|}$ is exactly the target of our algorithm.

Now we describe the process of the algorithm in detail. We represent a lattice point by its coordinates, and represent a convex lattice polygon by the sequence of coordinates of all lattice points contained in the convex lattice polygon. 

Our algorithm consists of three main parts: the main algorithm, the algorithm to determine whether the values of the unimodular invariants on two polygons are equal (denoted as COMPinv), and the algorithm to determine whether there is a unimodular transformation between two polygons (denoted as EQ).

Let $P$ be a convex lattice polygon. The unimodular invariants used in COMPinv algorithm are as follows:
\begin{enumerate}
	\item The number of vertices of $P$, denoted by $f_0(P)$.
	\item The number of lattice points on the boundary of $P$, denoted by $\mathrm{bound}(P)$.
	\item The area of P, denoted by $\mathrm{area}(P)$.
	\item Let $s_i(P), i=1, \cdots, f_0(P)$ be the number of lattice points on $i$th side of $P$. The sequence obtained by arranging $\{s_i(P)\}$ in ascending order is the fourth invariant, which is denoted by $\mathrm{sides}(P)$.
	\item Let $tr_i(P), i=1, \cdots, f_0(P)$ be the area of the adjacent triangle formed by the adjacent vertices of $i$th vertex of $P$, as shown in Figure $\ref{5}$. The sequence obtained by arranging $\{tr_i(P)\}$ in ascending order is the fifth invariant, which is denoted by $\mathrm{tr}(P)$.
\end{enumerate}

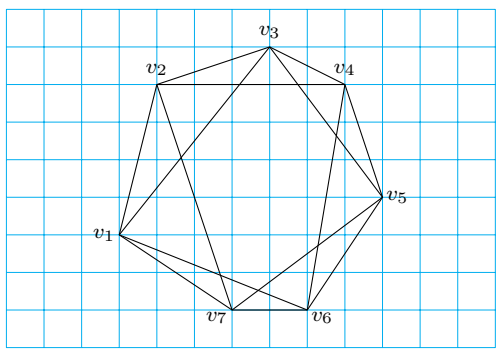
\begin{figure}
	\centering
	\begin{tikzpicture}
		\draw[step=.5cm,cyan,very thin] (-3.5,-2) grid (3,2.5);
		
		\draw (-2,-0.5) -- (-1.5,1.5) -- (0,2) -- (1,1.5) -- (1.5,0) -- (0.5,-1.5) -- (-0.5,-1.5) -- (-2,-.5);
		\draw (-2,-0.5) -- (0,2);
		\draw (-1.5,1.5) -- (1,1.5);
		\draw (0,2) -- (1.5,0);
		\draw (1,1.5) -- (0.5,-1.5);
		\draw (1.5,0) -- (-0.5,-1.5);
		\draw (-2,-0.5) -- (0.5,-1.5);
		\draw (-1.5,1.5) -- (-0.5,-1.5);
		
		\path[font=\fontsize{8}{8}\selectfont] (-2.2,-0.5) node (v1) {$v_1$};
		\path[font=\fontsize{8}{8}\selectfont] (-1.5,1.7) node (v2) {$v_2$};
		\path[font=\fontsize{8}{8}\selectfont] (0,2.2) node (v3) {$v_3$};
		\path[font=\fontsize{8}{8}\selectfont] (1,1.7) node (v4) {$v_4$};
		\path[font=\fontsize{8}{8}\selectfont] (1.7,0) node (v5) {$v_5$};
		\path[font=\fontsize{8}{8}\selectfont] (0.7,-1.6) node (v6) {$v_6$};
		\path[font=\fontsize{8}{8}\selectfont] (-0.7,-1.6) node (v7) {$v_7$};
	\end{tikzpicture}
	\caption{Vertices and Adjacent Triangles}
	\label{5}
\end{figure}


The first three invariants are referred to as first-level invariants, while the last two are called second-level invariants. In the COMPinv algorithm, the first step involves verifying the first-level invariants, and subsequently checking the second-level invariants only if the first-level invariant values are equal.
\begin{rem}\label{invext}
	The above invariants are proved to be effective in the case of small radius by experiments. In addition, each of the above invariants can be extended to the $n$-dimensional case.
\end{rem}

The specific processes of the main algorithm, COMPinv algorithm and EQ algorithm are shown in Algorithm \ref{themainalg}, Algorithm \ref{INVAR} and Algorithm \ref{EQ} respectively.

In Algorithm \ref{themainalg}, the function CALinv takes a polygon as input and outputs the values of the invariants on that polygon. Let $\mathrm{INVAR}[P_v]$ be the values of invariants on $P_v$, and let $\mathcal{P}_{inv}^{n-1}\coloneqq\{\mathrm{INVAR}[P]:P\in\mathcal{P}^{n-1}\}$, for $n=4,\cdots,\left|N_R\right|+1$.


Algorithm \ref{INVAR} takes two polygons as input. If the values of all invariants on these two polygons are found to be equal, the COMPinv algorithm outputs 1, otherwise, it outputs 0. Obviously, when COMPinv algorithm outputs 0, the two polygons must not be unimodularly equivalent. In Algorithm \ref{INVAR}, the MATLAB command \textsc{convhull} is utilized to obtain all the vertices or the lattice points on the boundary of a given convex latiice polygon. It should be noted that the command \textsc{convhull} will report an error if the convex hull of the input points is not two-dimensional, and we check the dimension of a given convex lattice polygon in Algorithm \ref{themainalg} by catching this particular error.

Algorithm \ref{EQ} takes two polygons as input and determines whether there exists a unimodular transformation between them. If such a transformation exists, the EQ algorithm outputs 1, otherwise it outputs 0. Because the unimodular transformation does not change $tr_{i}(P)$, we can determine the possible images of any vertex based on the area of adjacent triangles. Then, we can find the unimodular matrix. The specific process is as follows. Let $v^{(1)}_i$, $i=1,\cdots,n_1$ denote the vertices of $P_1$, where each $v^{(1)}_i$ and its adjacent vertices form a triangle having minimal area $s_{1}$ among all such adjacent triangles. Then we fix $v^{(1)}_1$ and move it to the origin. If $v^{(2)}_i, i=1,\cdots,n_2$ are the vertices of $P_2$ whose adjacent triangles' area is equal to the $s_{1}$, we move them to the origin in turn. $P_{1}$ and $P_{2}$ do the same translation respectively. Let $\mathbf{a}$, $\mathbf{b}$ denote the edge vectors adjacent to translated vertex $v^{(1)}_{1}$, and let $\mathbf{c_i}$, $\mathbf{d_{i}}$ denote the edge vectors adjacent to translated vertex $v^{(2)}_i, i=1,\cdots,n_2$, where the edge vector adjacent to a vertex $\mathbf{x}$ of a lattice polytope is a vector $\mathbf{v}\in \mathbb{Z}^{n}$ such that $\mathbf{v}=\mathbf{y}-\mathbf{x}$ for some vertex $\mathbf{y}$ for which conv$(\mathbf{x},\mathbf{y})$ is an edge. Then we check whether there is a unimodular matrix between edge vectors $\mathbf{a}$, $\mathbf{b}$ and edge vectors $\mathbf{c_i}$, $\mathbf{d_{i}}$ for each $i$. If it does not exist, the two convex lattice polygons are not unimodularly equivalent. If it exists, we check whether the unimodular matrix maps the vertices of $P_{1}$ to the corresponding vertices of $P_{2}$. If it does, then the two polygons are unimodularly equivalent, otherwise, they are not unimodularly equivalent.


\begin{rem}
   From the above discussion, it can be seen that Algorithms \ref{themainalg}, \ref{INVAR}, and \ref{EQ} can \emph{almost} be extended to higher-dimensional cases. The main potential challenge is that vertices can be sorted in a counterclockwise or clockwise order in the two-dimensional case, but this is not applicable in higher-dimensional cases. Consequently, identifying the correspondence between the vertices of two convex lattice polytopes would be more challenging, which may result in a significantly higher time complexity for Algorithm \ref{EQ} in higher-dimensional cases.
\end{rem}

\begin{algorithm}
	\caption{The main algorithm}\label{themainalg}
	\KwIn{$N_R$}
	\KwOut{$\mathcal{P}=\mathcal{P}^{3}\cup\cdots\cup\mathcal{P}^{|N_{R}|}$}\
	Set {$n\leftarrow\left|N_R\right|$, $\mathcal{P}^{|N_{R}|}\leftarrow\{N_{R}\}$, $\mathcal{P}^{i}=\emptyset,i=3,4,...,|N_{R}|-1$}\;
        $\mathrm{INVAR}[N_R]\leftarrow\mathrm{CALinv}(N_R)$\;
        Set $\mathcal{P}_{inv}^{|N_{R}|}\leftarrow\{\mathrm{INVAR}[N_R]\}$, $\mathcal{P}_{inv}^{i}=\emptyset,i=3,4,...,|N_{R}|-1$;
	
	\While{$n\geqslant 4$}{
		\For{$P\in \mathcal{P}^{n}$}{
			\For{$v\in vert(P)$}{
				$P_{v}\leftarrow (P\cap\mathbb{Z}^{2})\backslash\{v\}$\;
				
				\If{$\mathrm{dim}(P_{v})=2$}{
					$\mathrm{INVAR}[P_v]\leftarrow\mathrm{CALinv}(P_v)$\;
					\eIf{$\mathcal{P}^{n-1}$ is empty}{
						$\mathcal{P}^{n-1}\leftarrow\{P_{v}\}$\;
						$\mathcal{P}_{inv}^{n-1}\leftarrow\left\{\mathrm{INVAR}[P_v]\right\}$}
					{$j\leftarrow 1$, $\mathrm{EQidx}\leftarrow 0$\;
						\While{$1\leqslant j\leqslant |\mathcal{P}^{n-1}|$}{
							\uIf{$\mathrm{COMPinv}(\mathcal{P}^{n-1}_{inv}[j],\mathrm{INVAR}[P_v])==0$}{$j\leftarrow j+1$}
							\uElseIf{$\mathrm{EQ}(\mathcal{P}^{n-1}[j], P_{v})==1$}
							{$\mathrm{EQidx}\leftarrow 1$\;
								$j\leftarrow\left|\mathcal{P}^{n-1}\right|+1$\;}
							\Else{$j\leftarrow j+1$\;}
						}
						
						\If{$\mathrm{EQidx}==0$}{$\mathcal{P}^{n-1}\leftarrow\mathcal{P}^{n-1}\cup\{P_{v}\}$\;
					$\mathcal{P}_{inv}^{n-1}\leftarrow\mathcal{P}_{inv}^{n-1}\cup\left\{\mathrm{INVAR}[P_v]\right\}$}}
					
			}}
		}
		$n\leftarrow n-1$
	}
\end{algorithm}

\begin{algorithm}
	\caption{COMPinv algorithm}\label{INVAR}
	\KwIn{the values of the unimodular invariants on two convex lattice polygons $P_1$ and $P_2$}
	\KwOut{a logical variable indicating whether the values of the invariants on $P_1$ and $P_2$ are equal}
	Read $f_0(P_1)$, $\mathrm{bound}(P_1)$, $\mathrm{area}(P_1)$, $f_0(P_2)$, $\mathrm{bound}(P_2)$ and $\mathrm{area}(P_2)$\;
	Read $\mathrm{sides}(P_1)$, $\mathrm{tr}(P_1)$, $\mathrm{sides}(P_2)$ and $\mathrm{tr}(P_2)$\;
	$\mathrm{FirLev}(P_1)\leftarrow [f_0(P_1), \mathrm{bound}(P_1), \mathrm{area}(P_1)]$; $\mathrm{SecLev}(P_1)\leftarrow [\mathrm{sides}(P_1), \mathrm{tr}(P_1)]$\;
	$\mathrm{FirLev}(P_2)\leftarrow [f_0(P_2), \mathrm{bound}(P_2), \mathrm{area}(P_2)]$; $\mathrm{SecLev}(P_2)\leftarrow [\mathrm{sides}(P_2),\mathrm{tr}(P_2)]$\;
	\uIf{$\mathrm{FirLev}(P_1)\neq\mathrm{FirLev}(P_2)$}{\KwOut{0}}
	\uElseIf{$\mathrm{SecLev}(P_1)==\mathrm{SecLev}(P_2)$}{\KwOut{1}}
	\Else{\KwOut{0}}
\end{algorithm}

\begin{algorithm}
	\caption{EQ algorithm}\label{EQ}
	\KwIn{two convex lattice polygons $P_1$ and $P_2$}
	\KwOut{a logical variable indicating whether there is a unimodular transformation between $P_1$ and $P_2$}
	Call $\mathrm{tr}(P_1)$ and $\mathrm{tr}(P_2)$ to solve out $v^{(1)}_i$, $i=1,\cdots,n_1$ and $v^{(2)}_i$, $i=1,\cdots,n_2$\; 
	$\mathrm{AlterVert}_1\leftarrow v^{(1)}_1$\;
	Arrange $\mathrm{vert}(P_1)$ counterclockwise and let $\mathrm{AlterVert}_1$ be the first vertex in $\mathrm{vert}(P_1)$\;
	Each vertex in $\mathrm{vert}(P_1)$ minus $\mathrm{AlterVert}_1$, the resulting sequence is still denoted by $\mathrm{vert}(P_1)$\;
	\For{$i\leftarrow 1$ \KwTo $n_2$}{$\mathrm{AlterVert}_2\leftarrow v^{(2)}_i$\;
		Arrange $\mathrm{vert}(P_2)$ counterclockwise and let $\mathrm{AlterVert}_2$ be the first vertex in $\mathrm{vert}(P_2)$\;
		Each vertex in $\mathrm{vert}(P_2)$ minus $\mathrm{AlterVert}_2$, the resulting sequence is still denoted by $\mathrm{vert}(P_2)$\;
		$\mathrm{point^{1}_1}\leftarrow\mathrm{vert}(P_1)[n_1]$; $\mathrm{point^{1}_2}\leftarrow\mathrm{vert}(P_1)[2]$\;
		$\mathrm{point^{2}_1}\leftarrow\mathrm{vert}(P_2)[n_2]$; $\mathrm{point^{2}_2}\leftarrow\mathrm{vert}(P_2)[2]$\;
		$\mathrm{unimomatrix}_1\leftarrow$ the unimodular matrix mapping $\mathrm{point^{1}_1}$ and $\mathrm{point^{1}_2}$ to $\mathrm{point^{2}_1}$ and $\mathrm{point^{2}_2}$ respectively\;
		$\mathrm{unimomatrix}_2\leftarrow$ the unimodular matrix mapping $\mathrm{point^{1}_1}$ and $\mathrm{point^{1}_2}$ to $\mathrm{point^{2}_2}$ and $\mathrm{point^{2}_1}$ respectively\;
		$\mathrm{vert}(P_2)_F\leftarrow$ flip $\mathrm{vert}(P_2)$ head to tail so that $\mathrm{AlterVert}_2$ is the last vertex\;
		\uIf{$(\mathrm{unimomatrix}_1\ is\ an\ integer\ matrix) \& (\mathrm{unimomatrix}_1\ maps\ \mathrm{vert}(P_1)\ to\ \mathrm{vert}(P_2))$}{\KwOut{1}}
		\uElseIf{$(\mathrm{unimomatrix}_2\ is\ an\ integer\ matrix) \& (\mathrm{unimomatrix}_2\ maps\ \mathrm{vert}(P_1)\ to\ \mathrm{vert}(P_2)_F)$}{\KwOut{1}}
		\Else{\KwOut{0}}
	}
\end{algorithm}

\section{Results of experiments}
We use the algorithm proposed in the previous section to solve out the representative elements of all convex lattice polygons in $B_2$, $B_3$ and $B_4$ respectively. The numbers of different classes of convex lattice polygons which have different cardinalities in $B_2$, $B_3$ and $B_4$ are shown in Table \ref{tableb23} and \ref{tableb4}.

Let $R_w$ be the number of all representative elements with $|P|=w$ and contained in $B_R$, then $\{R_w:R_w\neq 0\}$ in the case of $R=2$, 3 and 4 are listed in the Table \ref{tableb23} and \ref{tableb4}. The data in Table \ref{tableb23} (only in the case of $B_3$) and Table \ref{tableb4} are plotted as scatter plots, i.e., Figure \ref{figtab3} and Figure \ref{figtab4}. $R_w$ is clearly a function of $w$ when $R$ is fixed. From these two figures, we can see that when $R=3$ and $4$, $R_w$ increases and then decreases as $w$ increases. It is natural to ask if this trend of $R_w$ is true for any other $R$. More generally, we can propose the following problem.


\begin{prob}\label{rwofw}
    Fixing $R$ and considering $R_w$ as a function of $w$, find an approximate function that can fit $R_w$ well when $R$ is large.
\end{prob}

We can also propose the dual problem of Problem \ref{rwofw} as follows, which can help us to understand the extent to which representative elements with $w$ lattice points can be concentrated in discs with different radii. That is exactly the value of Problem \ref{rwofw} and Problem \ref{rwofr}.

\begin{prob}\label{rwofr}
    Fixing $w$ and considering $R_w$ as a function of $R$, find an approximate function that can fit $R_w$ well when $w$ is large.
\end{prob}

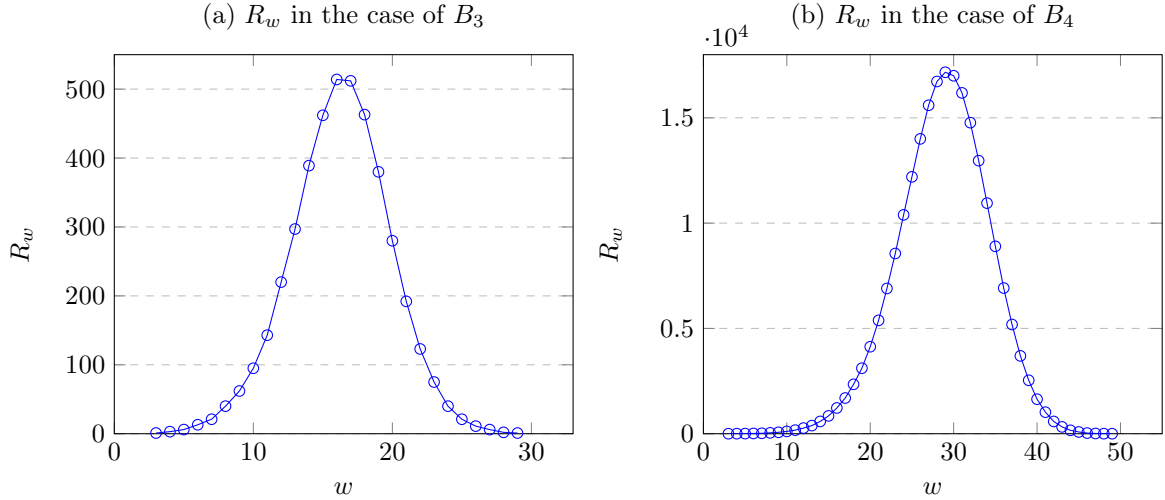
\begin{figure}
	\centering
	\subfigure{
        \label{figtab3}
	\begin{tikzpicture}
		\begin{axis}[
			title={(a) $R_w$ in the case of $B_3$},
			xlabel={$w$},
			ylabel={$R_w$},
			xmin=0, xmax=33,
			ymin=0, ymax=550,
			ymajorgrids=true,
			grid style=dashed,
			width=0.48\textwidth,
			]
			
			\addplot[
			color=blue,
			mark=o,
			]
			coordinates {
				(3,1)
				(4,3)
				(5,6)
				(6,13)
				(7,21)
				(8,40)
				(9,62)
				(10,95)
				(11,143)
				(12,220)
				(13,297)
				(14,389)
				(15,462)
				(16,514)
				(17,512)
				(18,463)
				(19,380)
				(20,280)
				(21,192)
				(22,123)
				(23,75)
				(24,40)
				(25,21)
				(26,11)
				(27,6)
				(28,2)
				(29,1)
			};
		\end{axis}
	\end{tikzpicture}}
	\subfigure{
        \label{figtab4}
	\begin{tikzpicture}
		\begin{axis}[
			title={(b) $R_w$ in the case of $B_4$},
			xlabel={$w$},
			ylabel={$R_w$},
			xmin=0, xmax=55,
			ymin=0, ymax=18000,
			ymajorgrids=true,
			grid style=dashed,
			width=0.48\textwidth,
			]
			
			\addplot[
			color=blue,
			mark=o,
			]
			coordinates {
				(3,1)
				(4,3)
				(5,6)
				(6,13)
				(7,21)
				(8,41)
				(9,67)
				(10,110)
				(11,170)
				(12,268)
				(13,386)
				(14,584)
				(15,846)
				(16,1223)
				(17,1695)
				(18,2346)
				(19,3111)
				(20,4132)
				(21,5383)
				(22,6898)
				(23,8558)
				(24,10392)
				(25,12198)
				(26,14001)
				(27,15598)
				(28,16726)
				(29,17165)
				(30,16998)
				(31,16185)
				(32,14771)
				(33,12967)
				(34,10950)
				(35,8899)
				(36,6918)
				(37,5186)
				(38,3696)
				(39,2537)
				(40,1640)
				(41,1023)
				(42,583)
				(43,324)
				(44,162)
				(45,83)
				(46,31)
				(47,12)
				(48,2)
				(49,1)
			};
		\end{axis}
	\end{tikzpicture}}
        \caption{Scatter plots of $R_w$ in the case of $B_3$ and $B_4$}
\end{figure}


\begin{table}[t]
	\centering
	\caption{The number of non-equivalent convex lattice polygons with different  cardinality in $B_2$ and $B_3$.}
	\label{tableb23}
	
 \begin{tabular}{|c|c|c|c|c|c|c|c|}
		\hline
		\multicolumn{3}{|c|}{$R$ (Radius)=2} & \multicolumn{5}{c|}{$R$ (Radius)=3}\\
		\hline
		$w$ (Cardinality) & $R_w$ & Sum of $R_w$ & $w$ (Cardinality) & $R_w$ & $w$ (Cardinality) & $R_w$ & Sum of $R_w$\\
		\hline
		3 & 1 & 75 & 3 & 1 & 17 & 512 & 4372\\
		\hline
		4 & 3 &    & 4 & 3 & 18 & 463 & \\
		\hline
		5 & 6 &    & 5 & 6 & 19 & 380 & \\
		\hline
		6 & 11 &    & 6 & 13 & 20 & 280 & \\
		\hline
		7 & 15 &    & 7 & 21 & 21 & 192 & \\
		\hline
		8 & 16 &    & 8 & 40 & 22 & 123 & \\
		\hline
		9 & 12 &    & 9 & 62 & 23 & 75 & \\
		\hline
		10 & 6 &    & 10 & 95 & 24 & 40 & \\
		\hline
		11 & 3 &    & 11 & 143 & 25 & 21 & \\
		\hline
		12 & 1 &    & 12 & 220 & 26 & 11 & \\
		\hline
		13 & 1 &    & 13 & 297 & 27 & 6 & \\
		\hline
		 &  &    & 14 & 389 & 28 & 2 & \\
		\hline
		 &  &    & 15 & 462 & 29 & 1 & \\
		\hline
		 &  &    & 16 & 514 &  &  & \\
		\hline
	\end{tabular}
\end{table}

\begin{table}[t]
	\centering
	\caption{The number of non-equivalent convex lattice polygons with different  cardinality in $B_4$.}
	\label{tableb4}
  \begin{tabular}{|c|c|c|c|c|c|c|c|c|}
		\hline
		\multicolumn{8}{|c|}{$R$ (Radius)=4}\\
		\hline
		$w$ (Cardinality) & $R_w$ & $w$ (Cardinality) & $R_w$ & $w$ (Cardinality) & $R_w$ & $w$ (Cardinality) & $R_w$\\
		\hline
		3 & 1 & 15 & 846 & 27 & 15589 & 39 & 2537\\
		\hline
		4 & 3 & 16 & 1223 & 28 & 16726 & 40 & 1640\\
		\hline
		5 & 6 & 17 & 1695 & 29 & 17165 & 41 & 1023\\
		\hline	
		6 & 13 & 18 & 2346 & 30 & 16998 & 42 & 583\\
		\hline
		7 & 21 & 19 & 3111 & 31 & 16185 & 43 & 324\\
		\hline
		8 & 41 & 20 & 4132 & 32 & 14771 & 44 & 162\\
		\hline
		9 & 67 & 21 & 5383 & 33 & 12967 & 45 & 83\\
		\hline
		10 & 110 & 22 & 6898 & 34 & 10950 & 46 & 31\\
		\hline
		11 & 170 & 23 & 8558 & 35 & 8899 & 47 & 12\\
		\hline
		12 & 268 & 24 & 10392 & 36 & 6918 & 48 & 2\\
		\hline
		13 & 386 & 25 & 12198 & 37 & 5186 & 49 & 1\\
		\hline
		14 & 584 & 26 & 14001 & 38 & 3696 & & \\	
		\hline
        \multicolumn{8}{|c|}{Sum of $R_w$=224901}\\
        \hline
	\end{tabular}
\end{table}

\section{Complexity Analysis of the Algorithms}
In this section we establish upper bounds on the number of arithmetic operations performed by Algorithm \ref{themainalg}, \ref{INVAR} and \ref{EQ}. Given an $R$ and the corresponding $B_R$, let $W$ be the number of lattice points contained in $B_R$ and  $M$ be the maximum number of vertices of convex lattice polygons contained in $B_R$. By (\ref{dingdianshang}) and \cite{balog}, we have
\begin{equation}\label{WandM}
	W\leqslant d_1 R^2,\quad M\leqslant d_2 R^{2/3},
\end{equation}
for some suitable constants $d_1$ and $d_2$. Let $N$ be the number of different classes of all  convex lattice polygons in $B_R$.

Firstly, we analyse the number of arithmetic operations performed by the function CALinv in Algorithm \ref{themainalg} that calculates the value of each invariant on a convex lattice polygon contained in $B_R$ as follows:
\begin{enumerate}
	\item As previously mentioned, for any convex lattice polygon, we use a MATLAB command \textsc{convhull} and take all lattice points contained in the polygon as input to obtain vertices and lattice points on the boundary of the polygon. By \cite{Hert}, the number of arithmetic operations performed by \textsc{convhull} to solve the convex hull of $n$ points is $O(n\log n)$. For any convex lattice polygon in $B_R$, the number of lattice points contained in it must be less than or equal to $W$. So the number of arithmetic operations required to obtain vertices and lattice points on the boundary of a convex lattice polygon in $B_R$ is $O(W\log W)$.
	\item The method used in our algorithm to calculate the area of a convex lattice polygon is as follows. Join one vertex of a polygon to the other vertices to obtain a kind of triangulation of the polygon, and then the area of the polygon $P$ is obtained by calculating the area of each small triangle and adding them together. The number of small triangles whose area is needed to be calculated is at most $|\mathrm{vert}(P)|-2$. Then the number of arithmetic operations needed to calculate the area of a convex lattice polygon is $O(M)$.
	\item Let $P$ be a convex lattice polygon in $B_R$. The number of arithmetic operations required to calculate the vector $(s_1(P),\cdots,s_{f_0(P)}(P))'$ is $O(M)$, where $f_0(P)$ is the number of vertices of $P$. This is because $P$ has at most $M$ sides. Then sorting all components of this vector in ascending order requires at most $O(M\log M)$ arithmetic operations. So it takes at most $O(M\log M)$ arithmetic operations to calculate the value of the fourth invariant on a given polygon $P$. The sorted vector is still denoted by $(s_1(P),\cdots,s_{f_0(P)}(P))'$.
	\item Let $P$ be a convex lattice polygon in $B_R$. The number of arithmetic operations required to calculate the vector $(tr_1(P),\cdots,tr_{f_0(P)}(P))'$ is $O(M)$, because $P$ has at most $M$ adjacent triangles. Similar to the analysis of the fourth invariant, taking into account the time complexity of sorting, it also takes at most $O(M\log M)$ arithmetic operations to calculate the value of the fifth invariant on a given polygon $P$. The sorted vector is still denoted by $(tr_1(P),\cdots,tr_{f_0(P)}(P))'$.
\end{enumerate}

In summary, the number of arithmetic operations needed to calculate the values of all invariants on a convex lattice polygon is $O(W \log W) + O(M\log M) = O(W \log W)$, that is, the function CALinv requires $O(W \log W)$ arithmetic operations.

Secondly, the Algorithm \ref{INVAR} only compares invariants instead of calculating them. Therefore, Algorithm \ref{INVAR} requires $O(M)$ arithmetic operations, because if the values of invariants are put together to form a new vector $(f_{0}(P),\mathrm{bound}(P),\mathrm{area}(P),s_1(P),\cdots,s_{f_0(P)}(P),tr_1(P),\cdots,tr_{f_0(P)}(P))$, then the length of the vector does not exceed $2M+3$.

Thirdly, we analyse the Algorithm \ref{EQ}, which aims to determine if there exists a unimodular transformation that maps $P_1$ to $P_2$. Since the function CALinv has already calculated $\mathrm{tr}(P_{1})$ and $\mathrm{tr}(P_{2})$, we can access them directly. After that, as previously mentioned process,  we fix a vertex of $P_1$, denoted by $v_{1}^{(1)}$, which, with its adjacent vertices, forms an adjacent triangle with minimal area $s_{1}$. For each vertex of $P_2$ whose adjacent triangle has the same area as $s_{1}$, it is possible to be the image of $v_{1}^{(1)}$. We translate $v_{1}^{(1)}$ and its possible images to origin. Then for each possible image of $v_{1}^{(1)}$, we verify whether there is a unimodular matrix between the adjacent edge vectors of $v_{1}^{(1)}$ and the adjacent edge vectors of the possible image, of which the number of arithmetic operations is $O(1)$. Next, we need to examine whether the unimodular matrix can map the other vertices of $P_1$ to those of $P_2$, where the process requires $O(M)$ arithmetic operations. This is because $P_1$ and $P_2$ have at most $M$ vertices. Besides, $v_{1}^{(1)}$ has at most $M$ possible images, we conclude that Algorithm \ref{EQ} requires $O(M\cdot M)=O(M^2)$ arithmetic operations.

Finally, we analyse the main algorithm. During the $i$th iteration of the loop that commences at line 4 in Algorithm \ref{themainalg} (hereafter denoted as the main loop), $\mathcal{P}^{W-i}$ is derived from $\mathcal{P}^{W-i+1}$, and the process can be divided into two parts. The first part involves shaving each vertex of each polygon in $\mathcal{P}^{W-i+1}$ in turn, resulting in a set of convex lattice polygons, denoted by $G^{W-i}$, in which each polygon contains $W-i$ lattice points. Then we calculate the invariants' values for each polygon obtained through shaving. Let 
\begin{equation*}
	\left|\mathcal{P}^{W-i+1}\right|=n_i,\quad \left|G^{W-i}\right|=n_{i+1}',
\end{equation*}
then it is easy to see
\begin{equation}
	n_{i+1}'\leqslant Mn_i.
\end{equation}
In addition, the number of times of shaving during the $i$th iteration of the main loop is also at most $Mn_i$.

The second part of the $i$th iteration of the main loop is to determine whether every two polygons in $G ^ {W-i} $ are equivalent. By doing so and retaining only one polygon for each equivalent class, we can obtain $\mathcal{P}^{W-i}$. It is evident that there are at most $C^{2}_{n_{i+1}'}$ determinations that will be made, where $C^{2}_{n_{i+1}'}$ is a combinatorial number.

Now we analyse these two parts respectively. In the first part, shaving a vertex from a convex lattice polygon requires at most $O(W)$ operations. This is because a convex lattice polygon is represented by the sequence of coordinates of all lattice points contained in it. Therefore, before shaving a vertex, we need to search for this vertex in the sequence of all lattice points contained in the polygon, which requires at most $W$ times. Recalling the number of arithmetic operations of the function CALinv in Algorithm \ref{themainalg} is $O(W\log W)$, the number of arithmetic operations performed by the first part is
\begin{equation}
	O(WMn_i +W \log Wn_{i+1}')=O((W+W\log W)Mn_i)=O(W \log W Mn_i).
\end{equation}


In the second part, we need to compare the values of the invariants on every two polygons in $G^{W-i}$. After that, we need to find the unimodular transformations between every two polygons in $G^{W-i}$ if they exist. Recalling Algorithm \ref{INVAR} and Algorithm \ref{EQ} require $O(M)$ and $O(M^{2})$ arithmetic operations respectively, the number of arithmetic operations performed by the second part is
\begin{equation}
	\begin{aligned}
		&O(C^{2}_{n_{i+1}'}(M^2 + M))\\
		=&O(n_{i+1}^{\prime2}(M^2 + M))\\
		=&O((M^2n^2_i)(M^2 + M))\\
		=&O((M^2n^2_i)M^2)\\
        =&O(M^4n_i^2).
	\end{aligned}
\end{equation}

Above all, the total number of arithmetic operations during the $i$th iteration of the main loop is 
\begin{equation}
	O(W \log WMn_i)+O(M^4n_i^2).
\end{equation}
Then the number of arithmetic operations performed by the main algorithm is 
\begin{equation}\label{complexityMWN}
	\begin{aligned}
		&O\left[\sum_{i=3}^W\left(W\log WMn_i+M^4n_i^2\right)\right]\\
		=&O(W\log W MN+M^4\sum_{i=3}^W n_i^2)\\
		=&O(W\log W MN+M^4N^2)\\
        =&O(M^4N^2).
	\end{aligned}
\end{equation}
By (\ref{shangjie}) and (\ref{WandM}), (\ref{complexityMWN}) can be written as 
\begin{equation}\label{complexityR}
	O(R^{8/3}2^{cR^{2/3}}).
\end{equation}

From (\ref{complexityMWN}) it can be seen that the complexity of our algorithm is a polynomial function in $M$ and $N$, and the degree of $N$ is only two, which is relatively low. This suggests that our algorithm is effective. However, Theorem \ref{bound} and (\ref{complexityR}) indicate that the problem of finding all non-equivalent convex lattice polygons in $B_R$ might be an NP-complete or NP-hard problem.

Furthermore, let $B^n_R$ be the $n$-dimensional closed ball of radius $R$ and centered at the origin of $\mathbb{E}^n$. By replacing $B_R$ in the above enumeration problem with $B^n_R$, we can get a more complex problem, which also has the potential to be an NP-complete or NP-hard problem. Such lattice-based problems with high complexity might be very promotive in the field of lattice cryptography and worth exploring.

\end{document}